\documentclass[12pt]{amsart}
\usepackage{amsfonts}
\usepackage{amsthm}
\usepackage{bbm}
\usepackage[margin=1.2in]{geometry}
\usepackage{xcolor}
\usepackage{mathrsfs}
\usepackage{cite}
\usepackage[
	colorlinks,
	linkcolor={black},
	citecolor={black},
	urlcolor={black}
]{hyperref}
\usepackage{comment}

\usepackage{enumitem}

\newtheorem{theorem}{Theorem}[section]
\newtheorem{lemma}[theorem]{Lemma}

\newtheorem{claim}{Claim}

\theoremstyle{definition}
\newtheorem{quest}[theorem]{Question}

\newtheorem{definition}[theorem]{Definition}

\newtheorem{remark}[theorem]{Remark}
\newtheorem*{claim*}{Claim}

\newtheorem{case}{Case}[theorem]

\newtheorem{subcase}{Subcase}[case]

\newenvironment{claimproof}[1][Proof of Claim]{\noindent \underline{#1.} }{\hfill$\diamondsuit$}
\theoremstyle{plain}

\newcommand{\bbC}{{\mathbbm{C}}}

\newcommand{\bbR}{{\mathbbm{R}}}
\newcommand{\bbZ}{{\mathbbm{Z}}}

\newcommand{\boH}{{\mathbf{H}}}
\newcommand{\boJ}{{\mathbf{J}}}
\newcommand{\boL}{{\mathbf{L}}}

\newcommand{\frg}{{\mathfrak{g}}}

\newcommand{\DSO}{\mathrm{DSO}}
\newcommand{\JAC}{\mathrm{JAC}}
\newcommand{\ODJM}{\mathrm{ODJM}}

\newcommand{\zeros}{{\mathbf{0}}}
\newcommand{\ones}{{\mathbf{1}}}

\newcommand{\tr}{{\mathrm{Tr}}}
\renewcommand{\Re}{{\mathrm{Re}}}

\newcommand{\idty}{{\mathbbm{1}}}
\newcommand{\cyclic}{{\mathrm{cyc}}}
\newcommand{\isotorus}{{\mathbf{T}}}
\newcommand{\spectrum}{{\mathrm{spec}}}
\newcommand{\energy}{{\mathcal{E}}}
\newcommand{\lcm}{{\mathrm{lcm}}}
\newcommand{\shifted}{{\clubsuit}}

\newcommand{\set}[1]{{\left\{ #1 \right\}}}

\numberwithin{equation}{section}

\title[Closed Gaps of 1D Discrete Operators]{On the Number of Closed Gaps of Discrete Periodic One-Dimensional Operators}

\author[A.\ Arroyo]{Andrew Arroyo}
\address{Department of Mathematics, Texas State University, San Marcos, TX 78666, USA}
\email{a\_a580@txstate.edu}

\author[F.\ Castro]{Faye Castro}
\address{Department of Mathematics, Texas State University, San Marcos, TX 78666, USA}
\email{veo21@txstate.edu}

\author[J.\ Fillman]{Jake Fillman}
\address{Department of Mathematics, Texas State University, San Marcos, TX 78666, USA}
\email{fillman@txstate.edu}

\thanks{Arroyo and Castro were supported in part by National Science Foundation Grant DMS 2213196. Fillman was supported in part by National Science Foundation Grant DMS 2213196 and Simons Foundation Collaboration Grant \# 711663.}

\begin{document}

\begin{abstract}
    From the general inverse theory of periodic Jacobi matrices, it is known that a periodic Jacobi matrix of minimal period $p \geq 2$ may have at most $p-2$ closed spectral gaps.
    We discuss the maximal number of closed gaps for one-dimensional periodic discrete Schr\"odinger operators of period $p$. 
    We prove nontrivial upper and lower bounds on this quantity for large $p$ and compute it exactly for $p \leq 6$. 
    Among our results, we show that a discrete Schr\"odinger operator of period four or five may have at most a single closed gap, and we characterize exactly which potentials may exhibit a closed gap.
For period six, we show that at most two gaps may close. 
    In all cases in which the maximal number of closed gaps is computed, it is seen to be strictly smaller than $p-2$, the bound guaranteed by the inverse theory.
    We also discuss similar results for purely off-diagonal Jacobi matrices.
\end{abstract}

\maketitle
\setcounter{tocdepth}{1}
\tableofcontents

\hypersetup{
	linkcolor={black!30!blue},
	citecolor={black!30!green},
	urlcolor={black!30!blue}
}

\section{Introduction}

\subsection{Setting}
Given $v \in \bbR^p$ and $a\in (0,\infty)^p=: \bbR_+^p$, the associated periodic Jacobi matrix $\boJ = \boJ_{a,v}:\ell^2(\bbZ) \to \ell^2(\bbZ)$ is defined by 
\begin{equation}
    [\boJ\psi](n)= A(n-1) \psi(n-1)+ V(n)\psi(n) + A(n)\psi(n+1),
\end{equation}
where $A,V:\bbZ \to \bbR$ are the $p$-periodic sequences satisfying $A(n) = a_n$ and $V(n)=v_n$ for $1\le n \le p$. In order to avoid trivialities that come from repetitions, we will always assume that $(a,v)$ is \emph{irreducible} in the sense that it has $p$ distinct cyclic shifts (which is equivalent to restricting attention to $A$ and $V$ such that $(A,V)$ has minimal period $p$). That is, defining $\cyclic:\bbR^p \to \bbR^p$ by $\cyclic(v) = (v_2,v_3,\ldots, v_p,v_1)$, we say that $(a,v) \in \bbR_+^p \times \bbR^p$ is \emph{irreducible} if
\[ (a,v), \ (\cyclic(a),\cyclic(v)), \ \ldots, \ (\cyclic^{p-1}(a), \cyclic^{p-1}(v)) \]
are pairwise distinct elements of $\bbR_+^p \times \bbR^p$, and we call it \emph{reducible} if it is not irreducible (with similar definitions of (ir)reducibility of $a$ and $v$).  

Periodic Jacobi matrices play an important role in spectral theory. On one hand, they arise quite naturally in the inverse spectral theory corresponding to finite-gap subsets of $\bbR$: for any finite-gap set $\Sigma \subseteq \bbR$ whose components have rational harmonic measure, there is a family (indeed, topologically a torus) of periodic Jacobi matrices with precisely that set as their common spectrum. See Appendix~\ref{sec:isotori} for a longer discussion and \cite{Simon2011Szego, Teschl2000Jacobi} for textbook discussions. From the perspective of direct spectral theory and mathematical physics, Jacobi matrices give one of the simplest models of a one-dimensional Hamiltonian with nearest-neighbor interactions. Indeed, the special case $a = \ones$, the vector of all ones, gives rise to \emph{discrete Schr\"odinger operators} (DSO), which have been extensively studied over the years; for background, we point the reader to the textbooks \cite{Bourgain2005greensfunction, CarmonaLacroix1990, CFKS, DF2022ESO, PasturFigotin1992}.

 It is known that the spectrum of $\boJ_{a,v}$ can be written as a union of nondegenerate closed intervals via the following procedure. For each $\theta \in \bbR$, define $J(\theta)=J_{a,v}(\theta)$ by
\begin{equation}
    J(\theta) = \begin{bmatrix}
        v_1 & a_1 &&& e^{-2\pi i\theta}a_p \\
        a_1 & v_2 & a_2 \\
        & \ddots & \ddots & \ddots\\
        && a_{p-2} & v_{p-1} & a_{p-1} \\
        e^{2\pi i\theta}a_p &&& a_{p-1} & v_p
    \end{bmatrix}.
\end{equation}
Let $\lambda_1(\theta) \leq \lambda_2(\theta) \leq \cdots \leq \lambda_p(\theta)$ denote the eigenvalues of $J(\theta)$ (counted with multiplicity). Defining
\[ \lambda_j^- = \min\{ \lambda_j(\theta): \theta \in [0,1] \}, \quad \lambda_j^+ = \max \{\lambda_j(\theta): \theta \in [0,1]\}, \]
we have the inequalities $\lambda_1^- < \lambda_1^+ \leq \lambda_2^- < \cdots \leq \lambda_p^- < \lambda_p^+$ and the following expression for the spectrum
\begin{equation}
    \spectrum(\boJ) = \bigcup_{j=1}^p [\lambda_j^-,\lambda_j^+].
\end{equation}
The intervals $[\lambda_j^-,\lambda_j^+]$ are called the \emph{bands} of the spectrum, and the intervals $(\lambda_j^+,\lambda_{j+1}^-)$ are called \emph{gaps}. One says that $\lambda \in \bbR$ is a \emph{closed gap} of $(a,v)$ if $\lambda = \lambda_j^+ = \lambda_{j+1}^-$ for some $1 \le j \le p-1$.

The famous  Borg--Hoschstadt theorem \cite{Borg1946Acta, Hochstadt1975LAA, Hochstadt1984LAA} asserts that if all spectral gaps collapse (i.e., $\lambda_j^+ = \lambda_{j+1}^-$ for all $1 \le j \le p-1$), then the diagonals and off-diagonals are constant; this immediately tells one that the maximal number of closed gaps for a Jacobi matrix of minimal period $p$ is at most $p-2$. In fact, within the class of all Jacobi matrices with minimal period $p$, the maximal number of closed gaps is precisely $p-2$, which can easily be deduced from the inverse theory (see Appendix~\ref{sec:isotori}). We are interested in what restrictions are placed on the structure of the spectrum if one restricts attention to the class of DSO.

 \begin{quest}
 How many \emph{closed} gaps can an irreducible $(a,v) \in \bbR_+^p \times \bbR^p$ have if it belongs to the class of discrete Schr\"odinger operators? 
 \end{quest}

 This work is related to and inspired by VandenBoom's work \cite{Vandenboom2018CMP}. More specifically, VandenBoom investigates broadly what sets whose components have rational harmonic measure can be spectra of discrete Schr\"odinger operators, and the question we are after here is of a related nature: how many connected components can the spectrum of a period-$p$ DSO have? We also discuss the related case of off-diagonal Jacobi matrices (ODJM), which correspond to $v=\zeros$.

 \subsection{Results}
To formulate results, let $\frg(a,v)$ denote the number of closed gaps\footnote{We will generally refer to closed gaps of the \emph{coefficient vector} $(a,v)$ rather than of the \emph{operator} $\boJ_{a,v}$, since we sometimes want to consider an operator as having period $kp$ where $p$ is its minimal period.} of $(a,v)$. When restricting to the class of discrete Schr\"odinger operators, we write $\boH_v = \boJ_{\ones,v}$ and $\frg_\DSO(v) := \frg(\ones,v)$. Similarly, we write $\boL_a = \boJ_{a,\zeros}$ and $\frg_\ODJM(a) = \frg(a,\zeros)$. We are then interested in the quantities:
\begin{align*}
    \frg_{\JAC}(p) &= \max\{\frg(a,v) : (a,v) \in \bbR_+^p \times \bbR^p \text{ is irreducible}\} \\ 
    \frg_{\DSO}(p) & = \max\{\frg_\DSO(v) : v \in \bbR^p \text{ is irreducible}\} \\ 
    \frg_{\ODJM}(p) & = \max\{\frg_\ODJM(a) : a \in \bbR_+^p  \text{ is irreducible}\}. 
\end{align*}
From the definitions, we see that $\frg(1)=0$ trivially in every case, so we focus on $p \geq 2$.  It is well known and not hard to show that $\frg_\DSO(v)=0$  for generic\footnote{In this case, a dense open set.} $v \in \bbR^p$ and similar reasoning shows $\frg_\ODJM(a)=0$ for generic $a \in \bbR_+^p$; compare \cite[Claim~3.4]{Avila2009CMP} and \cite[Lemma~2.1]{DamFilWan2023MANA}. 

 One can readily show that
\begin{equation} \label{eq:gJAC}
    \frg_{\JAC}(p) = p-2, \quad p \geq 2.
\end{equation}
Indeed, this follows imediately from the general inverse spectral theory of periodic Jacobi matrices.
For the reader's convenience, we give the proof of \eqref{eq:gJAC} in Appendix~\ref{sec:isotori}.

From \eqref{eq:gJAC}, we immediately note the upper bounds
\begin{equation} \label{eq:gDSOUB}
    \frg_{\DSO}(p), \ \frg_{\ODJM}(p) \leq  p-2, \quad p \geq 2.
\end{equation}
One naturally wonders how $\frg_\DSO$, $\frg_\ODJM$, and $\frg$ compare with one another aside from this bound. In particular, one naturally may wonder whether \eqref{eq:gDSOUB} is sharp. We will show that in general \eqref{eq:gDSOUB} is \emph{not sharp}.

\begin{theorem} \label{t:pgeq4}
 Let $p \geq 7$ be given.
 \begin{enumerate}[label={\rm (\alph*)}, itemsep=1ex]
 \item \label{pgeq4dso}  $\frg_\DSO(p) \geq 1$
 \item \label{pgeq4dsoub}  $\frg_\DSO(p) \leq p-3$ if $p \not\equiv 2 \ \mathrm{mod} \ 4$.
 \item \label{pgeq4odjm}  $\frg_\ODJM(p) \geq 1$ if $p$ is even and $\frg_\ODJM(p) \geq 2$ if $p$ is odd
 \item \label{pgeq4odjmub} $\frg_\ODJM(p) \leq p-3$.
 \end{enumerate}
\end{theorem}

\begin{remark}\mbox{}
 \begin{enumerate}[label={\rm (\alph*)}, itemsep=1ex]
\item We expect that the upper bounds from Theorem~\ref{t:pgeq4} are not sharp in general. We will see in the next theorem that the lower bounds can be improved inductively.
\item We also expect that the arithmetic assumption in part~\ref{pgeq4dsoub} is an artifact of the proof. It would be interesting to find an alternative approach in this case.
\item The theorem is only formulated for $p \geq 7$, since we compute the quantities in question exactly for $p \leq 6$ and give some additional information below. For instance, we will see that the lower bound from part~\ref{pgeq4dso} holds for all $p \geq 4$ and the upper bound from part~\ref{pgeq4dsoub} holds for all $p \geq 3$.
\end{enumerate}
\end{remark}
For larger periods, one can use an inductive construction to prove larger lower bounds:

\begin{theorem} \label{t:multiples}
    For all $p,k \geq 2$, and $\bullet \in \{\DSO,\ODJM\}$, 
    \begin{equation}
        \frg_\bullet(2kp) \geq \frg_\bullet(p)+(k-1)p.
    \end{equation}
\end{theorem}

\begin{remark}\mbox{}
\begin{enumerate}[label={\rm (\alph*)}, itemsep=1ex]
\item     It is not hard to see from the proof that one in fact has 
    \begin{equation}
        \frg_\bullet(mkp) \geq \frg_\bullet(p)+(k-1)p
    \end{equation}
    for all $m,p,k \geq 2$, and $\bullet \in \{\DSO,\ODJM\}$.
    \item This bound is also in general not sharp. For instance, taking $k=p=2$ gives the lower bound
    \[ \frg_\DSO(8) \geq \frg_\DSO(2) + (2-1)2 = 2. \]
    However, the reader can directly check that for any $\lambda\neq 0$, $v = (0,0,0,\lambda, 0,0,0,-\lambda)$ has closed gaps at $E = -\sqrt{2},0,+\sqrt{2}$, and hence
    \begin{equation}
        \frg_\DSO(8) \geq 3.
    \end{equation}
\end{enumerate}
\end{remark}

To supplement the abstract results, we also give exact computations for periods $p \leq 6$. The computations and their proofs suggest that the behavior of $\frg_\bullet$ may be subtle in general.

\begin{theorem} \label{t:123}
We have:
\begin{enumerate}[label={\rm (\alph*)}, itemsep=1ex]
    \item \label{p2dso} $\frg_\DSO(2)  = 0$,
    \item \label{p3dso} $\frg_\DSO(3)  = 0$,
    \item \label{p2odjm} $\frg_\ODJM(2) =  0$,
    \item \label{p3odjm} $\frg_\ODJM(3) = 0$.
\end{enumerate}
Equivalently, $\frg_\DSO(v) = 0$ for any irreducible $v$ in $\bbR^2$ or $\bbR^3$ and $\frg_\ODJM(a) = 0$ for any irreducible $a \in \bbR_+^2$ or $\bbR_+^3$.
\end{theorem}

\begin{remark}
    Let us point out that $\frg(a,v)=0$ for irreducible $(a,v)$ of period $2$ is well-known and follows immediately from the general theory (compare~\eqref{eq:gJAC}). It is included for completeness and because the calculations used to prove this directly furnish a basis for more elaborate computations later. We found the result for $p=3$ already to be intriguing, since $0 < \frg_\JAC(3) =1$, so the maximal number of closed gaps is strictly less than the theoretical upper limit provided by \eqref{eq:gJAC}.
\end{remark}

The next result shows $\frg_\bullet(4)=1$ for $\bullet \in \{\DSO,\ODJM\}$ and explicitly characterizes the coefficients that saturate the maximum. Again, we found this striking, since $4-2=2>1$, so the maximal number of gaps is again strictly smaller than $\frg_\JAC(4)=2$ in the restricted classes.
\begin{theorem} \label{t:4}
We have: 
\begin{enumerate}[label={\rm (\alph*)}, itemsep=1ex]
    \item \label{p4dso} $\frg_\DSO(4)= 1$. Moreover, if $v \in \bbR^4$ is irreducible, then  $\frg_\DSO(v)=1$ if and only if, up to an additive constant and a cyclic shift, $v$ has the form
    \begin{equation}
        v= (0,\lambda,0,-\lambda), \quad \lambda \neq 0.
    \end{equation}

\item \label{p4odjm} $\frg_\ODJM(4)=1$. Moreover, if $a \in \bbR_+^4$ is irreducible, then $\frg_\ODJM(a)=1$ if and only if
\begin{equation}
a_1 a_3=a_2a_4.
\end{equation}

\end{enumerate}
\end{theorem}

\begin{theorem} \label{t:5}
We have:
\begin{enumerate}[label={\rm (\alph*)}, itemsep=1ex]
\item \label{p5dso} $\frg_\DSO(5) = 1$. Furthermore, $\frg_\DSO(v) \geq 1$ for $v \in \bbR^5$ if and only if $v$ has one of the following forms, up to an additive constant and a cyclic shift:
\begin{enumerate}[label={\rm (\roman*)}, itemsep=1ex]
\item \label{item:p5dso+2closed} $v = (\lambda,\eta,\frac{\lambda+1}{\lambda\eta-1},\lambda\eta-1,\frac{\eta+1}{\lambda\eta-1})$ for some $\lambda,\eta$ such that $\lambda\eta \neq 1$.
\item \label{item:p5dso-2closed} $v = (\lambda,\eta,\frac{\lambda-1}{\lambda\eta-1},1-\lambda\eta,\frac{\eta-1}{\lambda\eta-1})$ for some $\lambda,\eta$ such that $\lambda\eta \neq 1$.
\end{enumerate}
\item \label{p5odjm} $\frg_\ODJM(5) = 2$. Furthermore, for $a \in \bbR_+^5$, one has $\frg_\ODJM(a) \geq 2$ if and only if, up to a multiplicative constant and a cyclic shift, $a$ has the form 
\begin{equation}
a = \left(\alpha,\beta, \sqrt{\frac{\alpha^2+\beta^2-1}{\alpha^2-1}}, \frac{\alpha \beta}{\sqrt{(\alpha^2-1)(\beta^2-1)}}, \sqrt{\frac{\alpha^2+\beta^2-1}{\beta^2-1}}\right).
\end{equation}
for some $\alpha,\beta$ such that either $\alpha>1$ and $\beta>1$ or $\alpha^2+\beta^2 < 1$.
\end{enumerate}
\end{theorem}
\begin{remark}\mbox{}
\begin{enumerate}[label={\rm (\alph*)}, itemsep=1ex]
\item We note that  $\frg_\DSO(5)$ and $\frg_\ODJM(5)$ are both strictly less than $\frg_\JAC(5)$. 
\item Case~\ref{item:p5dso+2closed} in Part~\ref{p5dso} corresponds to a gap associated with periodic boundary conditions closing while Case~\ref{item:p5dso-2closed} corresponds to a gap associated with antiperiodic boundary conditions closing.
\item Cases~\ref{item:p5dso+2closed} and \ref{item:p5dso-2closed} contain the special case in which $v$ is constant (and hence all gaps close). For instance, taking $\lambda = \eta =\varphi := (\sqrt{5}+1)/2$ or $\lambda = \eta =-\varphi^{-1}$ in Case~\ref{item:p5dso+2closed} gives a constant potential.
\item One has a similar classification in the off-diagonal Jacobi case. Concretely, by a reflection symmetry that we will describe later (see Lemma~\ref{lem:odjmReflect}), the spectral gaps of $a$ are reflection symmetric about the origin, and hence occur in matched pairs (since $5$ is odd). The case $\alpha,\beta>1$ in Part~\ref{p5odjm} corresponds to the ``inner'' gaps closing while $\alpha^2+\beta^2<1$ corresponds to the ``outer'' gaps closing.
\item The choices $\alpha = \beta = \varphi$ and $\alpha = \beta = \varphi^{-1}$ in Part~\ref{p5odjm} yield the cases of constant off-diagonals in the two settings.
\item Let us point out some symmetries that are not readily apparent from the expressions. Putting
\[ f_+(x,y) = \frac{x+ 1}{xy-1}, \quad g_+(x,y) = (y,f_+(x,y)), \]
one can check via direct computations that for $v$ as in Case~\ref{item:p5dso+2closed} one has $v_{j+1} = f_+(v_{j-1},v_j)$ for any $j$ for which the latter is defined. In particular, we point out that $g_+^5(x,y) = (x,y)$ for any $(x,y)$ for which $g_+^5$ is defined. A similar formalism holds for $f_-(x,y) = (x-1)/(xy-1)$ and in the ODJM case as well.
\end{enumerate}
\end{remark}

We conclude the computations for small periods with $p=6$. 
\begin{theorem} \label{t:6} We have:
\begin{enumerate}[label={\rm (\alph*)}, itemsep=1ex]
\item \label{p6dso}     $\frg_\DSO(6) = 2$ 
\item \label{p6odjm} $\frg_\ODJM(6)=3$.
\end{enumerate}
\end{theorem}

The rest of the paper is laid out as follows. We review some relevant background in Section~\ref{sec:background}, prove the exact computations for $\frg$ with $p \leq 6$ in Section~\ref{sec:comp}, and discuss the upper and lower bounds for larger $p$ in Section~\ref{sec:gen}. Appendix~\ref{sec:isotori} collects some important facts from the general inverse theory, and Appendix~\ref{sec:cx} discusses complex Jacobi matrices, in particular, why we formulate our results for Jacobi matrices with strictly positive off-diagonals.

\subsection*{Acknowledgements} We are grateful to Milivoje Luki\'{c} and Wencai Liu for comments on an earlier draft.

\section{Preliminaries} \label{sec:background}

Let us briefly recall important definitions and results from Floquet theory that will be used in the proofs of the main results. For textbook treatments, see \cite{Simon2011Szego, Teschl2000Jacobi}.

\begin{definition}
For  $s \in \bbR_+$, $t \in \bbR$, denote
\[B(s,t) 
:=  \frac{1}{s}\begin{bmatrix}
    t & -1\\ s^2 & 0
\end{bmatrix}. \]
If $(a,v) \in \bbR_+^p \times \bbR^p$, then the associated \emph{monodromy matrix} is defined by
\[ \Phi(E) = \Phi_{a,v}(E)
:= B(a_p,E- v_p) \cdots B(a_2, E-v_2) B(a_1, E- v_1) \]
and the \emph{discriminant} is given by
\begin{equation}
    D(E) = D_{a,v}(E) = \tr\,\Phi_{a,v}(E).
\end{equation}
\end{definition}

The significance of $B$ (and hence of $\Phi$) comes from noting that if $\boJ_{a,v}u = zu$ for a sequence $u:\bbZ \to \bbC$ and a scalar $z \in \bbC$, then
\begin{equation}
    \begin{bmatrix}
        u(n+1) \\ A(n)u(n)
    \end{bmatrix}
    = B(A(n), z - V(n))\begin{bmatrix}
        u(n) \\ A(n-1)u(n-1)
    \end{bmatrix}
\end{equation}
for every $n \in \bbZ$.

The spectrum of $\boJ_{a,v}$ is then given by
\begin{equation}
    \spectrum(\boJ_{a,v}) = \set{E \in \bbR : -2 \le D_{a,v}(E) \le 2}
\end{equation}
and furthermore $E$ is the location of a closed spectral gap of $(a,v)$ if and only if $\Phi_{a,v}(E) = \pm \idty$, where $\idty$ denotes the identity matrix.

For the off-diagonal Jacobi case, we need to leverage a suitable reflection symmetry. This is well-known and included to keep the paper self-contained.

\begin{lemma} \label{lem:odjmReflect}
    Consider $a \in \bbR_+^p$ and let $L_a(\theta) = J_{a,\zeros}(\theta)$ denote the corresponding Floquet matrix. 
    \begin{enumerate}[label={\rm (\alph*)}, itemsep=1ex]
    \item If $p$ is even, then $L_a(\theta)$ is unitarily equivalent to $-L_a(\theta)$. 
    \item If $p$ is odd, then $L_a(\theta)$ is unitarily equivalent to $-L_a(\theta+\frac{1}{2})$.
    \item If $p$ is even, then $E =0$ either lies in an open spectral gap, or it is the location of a closed spectral gap. 
    
    \item If $p$ is odd, then $E=0$ lies in the interior of a spectral band.
    \end{enumerate}

    In particular, $a$ has a closed gap at energy $E \in \bbR$ if and only if it has a closed gap at $-E$.
\end{lemma}

\begin{proof}
    All statements follow by conjugating $L_a(\theta)$ with the unitary matrix $$U = \mathrm{diag}(-1,1,-1,1,\ldots,(-1)^p),$$
    that is, $[Uv]_n = (-1)^n v_n$.
\end{proof}

When we discuss the case of discrete Schr\"odinger operators, we often abbreviate
\[
M(t) := B(1,t)=
\begin{bmatrix}
    t & -1 \\ 1 & 0
\end{bmatrix},
\]
and, given $v \in \bbR^p$ and $E \in \bbR$, we write
\begin{align*} \Phi_v(E)
:= \Phi_{\ones,v}(E)
& = M(E-v_p) \cdots M(E-v_1) \\
& = \Phi_{v_p}(E) \cdots \Phi_{v_1}(E). \end{align*}
For the off-diagonal Jacobi case, we similarly write
\begin{align*} 
\Psi_a(E) := \Phi_{a,\zeros}(E)
& = B(a_p,E) \cdots B(a_1,E) \\
& = \Psi_{a_p}(E) \cdots \Psi_{a_1}(E).  \end{align*}
We also will sometimes use free monoid notation, for instance writing $v=v_1v_2\cdots v_p$ for a general element of $\bbR^p$. For $v \in \bbR^p$, $w \in \bbR^q$, we then write $vw = v_1\cdots v_p w_1 \cdots w_q \in \bbR^{p+q}$ for their concatenation. From this point of view, the map from $v$ to $\Phi_v(E)$ at fixed $E$ is an \emph{anti-homomorphism} in the sense that
\begin{equation}
    \Phi_{vw}(E) = \Phi_w(E)\Phi_v(E).
\end{equation}

\begin{remark} \label{rem:shift} There are a few ideas that we use to simplify the calculations.  First, one may freely rewrite the identity $\Phi_v(E)=\pm \idty$ as
\begin{equation} 
\Phi_{v_1\cdots v_m}(E) \mp [\Phi_{v_{m+1}\cdots v_p}(E)]^{-1} = 0 \end{equation}
for a suitable choice of $m$; generally, we choose $m = \lceil p/2\rceil$, which has the effect of reducing the complexity of the system(s) of polynomial equations under consideration by (approximately) a factor of two. Second, one can note that
\[ \Phi_{\cyclic(v)}(E) = \Phi_{v_1}(E) \Phi_v(E) [\Phi_{v_1}(E)]^{-1}, \]
so $\Phi_v(E) = \pm \idty$ if and only if $\Phi_{\cyclic(v)}(E) = \pm \idty$. Thus, for any identity that is obtained from the assumption of a closed gap, one may cyclically permute the variables to obtain $p-1$ additional relations. Concretely, if one obtains a relation that is \emph{linear} in the variables of $v$, then one can immediately deduce that $v$ lies in the kernel of an explicit circulant matrix. This is put to use explicitly in the proof of Theorem~\ref{t:5}, and implicitly in the proofs of Theorems~\ref{t:6} and \ref{t:pgeq4}.  Third and finally, $v$ has a closed gap at energy $E$ if and only if $v+c\ones$ has a closed gap at energy $E+c$. Thus, we can use an additive constant to shift the gaps, which is sometimes useful to simplify the calculations or exploit a symmetry argument. Similarly, in the ODJM case, one can use a multiplicative constant to scale the locations of the gaps; that is, $a$ has a closed gap at energy $E$ if and only if $ca$ has a closed gap at energy $cE$.
\end{remark}


\section{Explicit Computations for Small Periods} \label{sec:comp}

\begin{proof}[Proof of Theorem~\ref{t:123}]
\ref{p2dso}
 As noted above, the statement in this case already follows from the abstract theory. We work it out explicitly since the calculations here are helpful for later cases. Let $v=(v_1,v_2) \in \bbR^2$ be given, and observe
\begin{align} \nonumber
    \Phi_v(E) 
    & = \begin{bmatrix}
        E - v_2 & -1 \\ 1 & 0
    \end{bmatrix}
    \begin{bmatrix}
        E - v_1 & -1 \\ 1 & 0
    \end{bmatrix}\\
    \label{eq:p2Phiv}
    & = 
    \begin{bmatrix}
        (E-v_1)(E-v_2)-1 & -(E-v_2)\\
        (E-v_1) & -1
    \end{bmatrix}.
\end{align}
This immediately implies $\Phi_v(E)= \idty$ is impossible. 
If $\Phi_v(E) = -\idty$, then adding the $(1,2)$ and $(2,1)$ entries of \eqref{eq:p2Phiv} gives
$$ 0= (E-v_1)-(E-v_2) = v_2-v_1,$$
which implies that $v$ is reducible. Thus, no gaps may close for irreducible $v \in \bbR^2$.
\bigskip

\ref{p3dso}
Let $v=(v_1,v_2,v_3)$ be given, and assume that $v$ has a closed gap at $E \in \bbR$. Rearranging $\Phi_v(E) = \pm \idty$ as suggested in Remark~\ref{rem:shift}, we get
\begin{align}
\Phi_{v_1v_2}(E) = \pm [\Phi_{v_3}(E)]^{-1}, 
\end{align}
which, using \eqref{eq:p2Phiv}, gives
 \begin{equation}
   \begin{bmatrix}
        (E-v_1)(E-v_2)-1 & -(E-v_2)\\
        (E-v_1) & -1
    \end{bmatrix}
    = \pm
    \begin{bmatrix}
        0 & 1 \\ -1 & E-v_3
    \end{bmatrix}.
 \end{equation}
 By examining the $(1,2)$, $(2,1)$, and $(2,2)$ entries, we get
 \[ E-v_1 = E-v_2 = E-v_3 = \mp 1, \]
 leading to $v_1=v_2=v_3$, which implies that $v$ is not irreducible. Therefore, there are no irreducible $v \in \bbR^3$ with a closed gap.
\bigskip

\ref{p2odjm} Let $a=(a_1,a_2) \in \bbR_+^2$ be given, and observe
\begin{align} \nonumber
    \Psi_a(E) 
    & = \frac{1}{a_1a_2} \begin{bmatrix}
        E & - 1 \\ a_2^2 & 0
    \end{bmatrix}
    \begin{bmatrix}
        E & - 1 \\ a_1^2 & 0
    \end{bmatrix}\\
    \label{eq:p2Psia}
    & = 
    \frac{1}{a_1a_2}
    \begin{bmatrix}
        E^2-a_1^2 & -E\\
        a_2^2 E & -a_2^2
    \end{bmatrix}.
\end{align}
Then $\Psi_a(E) = \idty$ is impossible and $\Psi_a(E) = - \idty$ forces $E = 0$, and in turn $a_1=a_2$, which implies that $a$ is reducible. 
\medskip

\ref{p3odjm}
Let $a=(a_1,a_2,a_3) \in \bbR_+^3$ be given,\footnote{At the risk of being repetitive, we point out that there is also a ``soft'' argument in this case as well. By reflection symmetry, if one gap closes, then both gaps close and hence $a$ is constant by the Borg--Hochstadt theorem \cite{Borg1946Acta, Hochstadt1975LAA, Hochstadt1984LAA}. Indeed (the generalization of) this argument is precisely the basis of (part of) the proof of Theorem~\ref{t:pgeq4}.\ref{pgeq4odjmub}.} and assume $E \in \bbR$ is  a closed gap of $a$. Rearranging $\Psi_a(E) = \pm \idty$ using \eqref{eq:p2Psia} gives
 \begin{equation}
   \frac{1}{a_1a_2}
    \begin{bmatrix}
        E^2-a_1^2 & -E\\
        a_2^2 E & -a_2^2
    \end{bmatrix}
    = \pm
    \frac{1}{a_3}
    \begin{bmatrix}
        0 & 1 \\ -a_3^2 & E
    \end{bmatrix}.
 \end{equation}
Examining the  $(1,2)$, $(2,1)$, and $(2,2)$ entries yields
\[ \frac{a_1a_2}{a_3} = \frac{a_1a_3}{a_2} = \frac{a_2a_3}{a_1} = \mp E,  \]
which implies $a_1=a_2=a_3$ and hence that $a$ is reducible.
 Therefore, there are no irreducible $a \in \bbR_+^3$ with a closed gap.
\end{proof}

\begin{proof}[Proof of Theorem~\ref{t:4}]
\ref{p4dso}     Given $v \in \bbR^4$, assume that $E$ is a closed gap of $v$. We will show that $\Phi_v(E)=-\idty$ implies $v$ is reducible and $\Phi_v(E)=\idty$ forces $v$ to be as claimed in the theorem statement.
\setcounter{case}{0}
        \begin{case}
            \textbf{\boldmath $\Phi_v(E) = -\idty$.}
        \end{case}
        As discussed in Remark~\ref{rem:shift}, we rewrite $\Phi_v(E)=-\idty$ as $\Phi_{v_1v_2}(E) =- [\Phi_{v_3v_4}(E)]^{-1}$.
Writing this out with \eqref{eq:p2Phiv} yields
 \begin{equation}
   \begin{bmatrix}
        (E-v_1)(E-v_2)-1 & -(E-v_2)\\
        (E-v_1) & -1
    \end{bmatrix}
    = -
    \begin{bmatrix}
         -1 & (E-v_4)\\
        -(E-v_3) & (E-v_3)(E-v_4)-1
    \end{bmatrix}.
 \end{equation}
 Looking at the $(1,2)$ and $(2,1)$ entries gives $v_1 = v_3$ and $v_2=v_4$, so this case implies $v$ is reducible.
        
\begin{case}
            \textbf{\boldmath $\Phi_v(E) = \idty$.}
        \end{case}
        By adding a constant to $v$, we may assume without loss that $E=0$. Rearranging the identity $\Phi_v(E)=\idty$ as above and setting $E=0$ yields
 \begin{equation}
   \begin{bmatrix}
        v_1 v_2-1 & v_2\\
        -v_1 & -1
    \end{bmatrix}
    = 
    \begin{bmatrix}
         -1 & -v_4\\
        v_3 & v_3 v_4-1
    \end{bmatrix}.
 \end{equation}
From this we see that one has $v_1=-v_3$, $v_2=-v_4$, and $v_1v_2=0$. Thus, up to a cyclic shift, $v$ has the desired form, proving the ``only if'' part of \ref{p4dso}. A direct computation verifies that $v$ has a closed gap at $E = 0$ whenever  $v$ has the form $(0,\lambda ,0,-\lambda)$ for some $ \lambda $. Since the presence of closed gaps is preserved by cyclic shifts and additive constants, this proves the ``if'' direction.\bigskip

\ref{p4odjm} Let $a \in \bbR_+^4$ be given, and suppose $a$ has a closed gap at $E \in \bbR$.
\setcounter{case}{0}
 \begin{case}
     \textbf{\boldmath $\Psi_a(E)=-\idty$}
 \end{case}
Using \eqref{eq:p2Psia}, we can rewrite this as
 \[ \frac{1}{a_1a_2}
    \begin{bmatrix}
        E^2-a_1^2 & -E\\
        a_2^2 E & -a_2^2
    \end{bmatrix}
    = - \frac{1}{a_3a_4}
    \begin{bmatrix}
        -a_4^2 & E\\
       - a_4^2 E &  E^2-a_3^2
    \end{bmatrix}\]
    First, by looking at the $(1,2)$ entries, we see that either $E = 0$ or $a_1a_2=a_3a_4$. However, $E = 0$ is impossible; indeed, if $E=0$, the $(1,1)$ entry of the left hand side is negative and the corresponding entry of the right hand side is positive. Thus, we have $E \neq 0$ and
    \begin{equation} \label{eq:a12=a34} a_1a_2=a_3a_4.
    \end{equation}
    Using this and examining the $(2,1)$ entries, we deduce $a_1 = a_3$ and $a_2=a_4$,  which implies that $a$ is reducible.

Thus, neither of the gaps corresponding to $\tr\, \Psi= -2$ may close for $p=4$ and irreducible $a \in \bbR_+^4$.

    \begin{case}
     \textbf{\boldmath $\Psi_a(E)=\idty$}
 \end{case}
Rewriting as above gives us
 \[ \frac{1}{a_1a_2}
    \begin{bmatrix}
        E^2-a_1^2 & -E\\
        a_2^2 E & -a_2^2
    \end{bmatrix}
    =  \frac{1}{a_3a_4}
    \begin{bmatrix}
        -a_4^2 & E\\
       - a_4^2 E &  E^2-a_3^2
    \end{bmatrix}\]
If $E \neq 0$, then the $(1,2)$ entries of each side have opposite signs, so we must have $E =0$. Looking at the $(1,1)$ entries gives $a_1/a_2 = a_4/a_3$, as desired. Moreover, one can check from these calculations that if $a_1a_3=a_2a_4$, then $a$ has a closed gap at $E=0$.
\end{proof}

\begin{proof}[Proof of Theorem~\ref{t:5}]
\ref{p5dso} First, note that if $v$ has the form given in Case~\ref{item:p5dso+2closed} or Case~\ref{item:p5dso-2closed}, then $v$ enjoys a closed gap at $E=0$ by direct computations: for instance, if $v$ is as in Case~\ref{item:p5dso+2closed},  then applying \eqref{eq:p2Phiv} and \eqref{eq:p3Phiv} with $E =0$, we arrive at
\begin{align*}
\Phi_{v_1v_2v_3}(0) - [\Phi_{v_4v5}(0)]^{-1}
& = \begin{bmatrix}
        -v_1v_2v_3+ v_1 +v_3  & -v_2 v_3+1 \\
        v_1 v_2 -1 & v_2
    \end{bmatrix}
- \begin{bmatrix}
        -1 & - v_5 \\
        v_4 & v_4v_5 -1
    \end{bmatrix} \\
& = \begin{bmatrix}
-\frac{\lambda\eta(\lambda + 1)}{\lambda \eta-1} + \lambda + \frac{\lambda+1}{\lambda \eta-1} + 1  & -\frac{\eta(\lambda+1)}{\lambda \eta - 1}+\frac{\eta + 1}{\lambda \eta - 1}+1 \\
\lambda \eta-(\lambda \eta-1)-1 & -(\eta+1) + \eta + 1
\end{bmatrix} \\
& = 0 .
\end{align*}
Since such $v$ are irreducible whenever $\lambda \neq \eta$, we deduce also that $\frg_\DSO(5) \geq 1$.

For the other inequality, let us show that if $v$ has two or more closed gaps, then $v$ is constant (in particular, reducible). To that end, assume $v$ has two closed gaps. This implies that $v$ either has two closed gaps satisfying $\tr(\Phi_v(E))=+2$, two closed gaps satisfying $\tr(\Phi_v(E))=-2$, or one of each flavor. We will show that any of these situations forces $v$ to be a constant vector.

\setcounter{case}{0}
\begin{case} \textbf{\boldmath $\Phi_v(E_1)=\Phi_v(E_2)=\idty$ for some $E_1\neq E_2$.}
\end{case}
As before, we may rewrite this as $\Phi_{v_1v_2v_3}(E_k)-[\Phi_{v_4v_5}(E_k)]^{-1} = 0$ for $k=1,2$. Computing with \eqref{eq:p2Phiv}, we have\small
\begin{align}
    \nonumber
    \Phi_{v_1v_2v_3}(E) 
    & = \begin{bmatrix}
        E - v_3 & -1 \\ 1 & 0
    \end{bmatrix}
    \begin{bmatrix}
        (E-v_1)(E-v_2)-1 & -(E-v_2)\\
        (E-v_1) & -1
    \end{bmatrix}
    \\
    \label{eq:p3Phiv}
    & = 
    \begin{bmatrix}
        (E-v_1)(E - v_2)(E - v_3) - (E - v_1) - (E-v_3) & -(E-v_2)(E-v_3)+1 \\
        (E-v_1)(E-v_2)-1 & - (E-v_2)
    \end{bmatrix}
\end{align}
\normalsize
Putting together \eqref{eq:p2Phiv}, \eqref{eq:p3Phiv}, and the assumption $\Phi_{v_1v_2v_3}(E_k) - [\Phi_{v_4v_5}(E_k)]^{-1}=0$, we get
\begin{equation}
\begin{bmatrix}
    * & * \\ * & -(E_k-v_2) - ((E_k-v_4)(E_k-v_5)-1)
\end{bmatrix} = 0,
\end{equation}
that is
\begin{equation}
    -E_k^2+(v_4+v_5-1)E_k +1  + v_2 -v_4v_5 = 0.
\end{equation}
Recalling the discussion in Remark~\ref{rem:shift}, we may cyclically permute the variables and thus we have
\begin{equation} \label{eq:p5diff1entry22}
    -E_k^2 +(v_{j-1} + v_{j-2}-1)E_k + 1+v_{j+1}-v_{j-1}v_{j-2} = 0, \quad 1\le j \le 5, \ k =1,2.
\end{equation}
Two quadratics with the same roots and the same leading coefficient must have all coefficients the same, so
 (applying this to $j$ and $j+1$), we see
\begin{equation}
    v_{j-1}+v_{j-2}=v_{j}+v_{j-1},
\end{equation}
leading to $v_{j-2}=v_j$ for all $j$ and thus,
\[ v_1 = v_3 = v_5 = v_2 = v_4, \]
that is, $v$ is constant.

\begin{case} \textbf{\boldmath $\Phi_v(E_1)=\Phi_v(E_2) = -\idty$ for some $E_1\neq E_2$.}
\end{case}
This case is similar to the previous one, but instead one considers $\Phi_{v_1v_2v_3}(E) + [\Phi_{v_4v_5}(E)]^{-1}$
and computes its $(2,2)$ entry to see that
\begin{equation} \label{eq:p5sum1entry22}
    E_k^2 -(v_{j-1} + v_{j-2} + 1)E_k -1+   v_{j+1} + v_{j-1} v_{j-2} = 0, \quad 1 \le j \le 5, \ k=1,2.
\end{equation}

The final case is somewhat more involved.

\begin{case} \textbf{\boldmath $\Phi_v(E_1)=-\Phi_v(E_2) = \idty$ for some $E_1\neq E_2$.}
\end{case}

By adding a constant to $v$, we may assume that $E_1 = -E_2 =: E$. Apply \eqref{eq:p5diff1entry22} with $E_k=E$, \eqref{eq:p5sum1entry22} with $E_k=-E$ and add them together to get
\begin{equation}
    E(v_{j-1}+v_{j-2})+v_{j+1}=0, \quad 1 \le j \le 5.
\end{equation}
Since this holds for every $j$, we see that $v$ lies in the kernel of the circulant matrix
\begin{equation}
Q(E):=
    \begin{bmatrix}
     E & E & 0 & 1 & 0  \\ 
     0 & E & E & 0 & 1  \\
     1 & 0 & E & E & 0  \\
     0 & 1 & 0 & E & E  \\
     E & 0 & 1 & 0 & E
    \end{bmatrix}
\end{equation}
We compute the determinant of $Q(E)$ and factor to get
\begin{equation}
    \det (Q(E)) = (2E+1)(E^2+E-1).
\end{equation}
In particular, we observe that the kernel of $Q$ is trivial unless $E = -1/2$, $E = -\varphi$ or $E = 1/\varphi$, where $\varphi = \frac{1}{2}(\sqrt{5}+1)$ denotes the golden ratio. When $E=-1/2$, the kernel of $Q$ is one-dimensional and spanned by $v=\ones$. Thus, for any $E \neq -\varphi, \varphi^{-1}$, we can already see that $v$ is reducible.

\begin{subcase}
    \textbf{\boldmath $E = -\varphi$.}
\end{subcase}

One readily computes that $\ker Q(-\varphi)$ is two-dimensional and spanned by the vectors $v^\bullet \in \bbR^5$
\[ v^{\cos}_n = 4\cos(4\pi (n-1) / 5), \quad v^{\sin}_n = \sqrt{8}\sin(4\pi (n-1)/5), \quad 1\le n \le 5. \]
(The $n-1$ is chosen to make $v_1$ correspond to $\cos(0)=1$ and $\sin(0)=0$, and the prefactors are chosen to simplify a few fractions).
Thus, we must have $v = av^{\cos} + bv^{\sin}$ for scalars $a,b \in \bbR$, so we may write
\begin{align*}
    v& = \bigg
(4a, 
-a(\sqrt{5} + 1) + b\sqrt{5-\sqrt{5}}, 
a(\sqrt{5}-1) - b\sqrt{5  + \sqrt{5}}, \ldots \\
& \qquad\qquad \ldots 
a(\sqrt{5}-1)+b\sqrt{5+\sqrt{5}}, 
-a(\sqrt{5}+1)-b \sqrt{5-\sqrt{5}}\bigg)
\end{align*} Recalling that $\Phi_v(E)=\idty$ and $\Phi_v(-E)=-\idty$ by assumption, we have
\[\Phi_v(-\varphi) = \idty.\]
Rewriting this as
\[\Phi_{v_1v_2v_3}(-\varphi) - [\Phi_{v_4v_5}(-\varphi)]^{-1} = 0\]
and computing the entries (using $\varphi^2=\varphi+1$ to simplify) gives us
\begin{align*}
    0
    & = [\Phi_{v_1v_2v_3}(-\varphi) - [\Phi_{v_4v_5}(-\varphi)]^{-1}]_{12} \\
    & = -(-\varphi-v_2)(-\varphi-v_3)+1 - (-\varphi-v_5) \\
    & = -(v_2+v_3)\varphi-v_2v_3  + v_5.
\end{align*}
Noting that $5\pm \sqrt{5} = 2\sqrt{5}\varphi^{\pm 1}$, we can see
\begin{align*}
    v_5-(v_2+v_3)\varphi 
    & =\left(-a(\sqrt{5}+1)-b \sqrt{5-\sqrt{5}}\right)\\
    & \qquad -\left( -a(\sqrt{5} + 1) + b\sqrt{5-\sqrt{5}} + 
a(\sqrt{5}-1) - b\sqrt{5  + \sqrt{5}}\right)\varphi \\
& = 0,
\end{align*}
and thus we arrive at
\begin{align}
\nonumber
    0
    & =-v_2v_3  \\
    \nonumber
   & = -\left(-a(\sqrt{5} + 1) + b\sqrt{5-\sqrt{5}} \right)
   \left( 
a(\sqrt{5}-1) - b\sqrt{5  + \sqrt{5}} \right)\\
\label{eq:p5:Phirearranged12}
& = 4a^2 - 4\sqrt{5+\sqrt{5}}\,ab +2\sqrt{5}\, b^2
\end{align}
In a similar way, looking at the $(2,2)$ entry gives
\begin{equation} \label{eq:p5:Phirearranged22}
     0 = 4a^2 + 4\sqrt{5+\sqrt{5}}\,ab +2\sqrt{5}\, b^2
\end{equation}
Adding \eqref{eq:p5:Phirearranged12} and \eqref{eq:p5:Phirearranged22} gives
\[ 8a^2 + 4\sqrt{5} \, b^2  = 0, \]
and thus $a=b=0$, which implies $v = \zeros$ and hence is reducible.

\begin{subcase}
    \textbf{\boldmath $E = 1/\varphi$.} \end{subcase}
    Similar to before, $\ker Q(1/\varphi)$ is two-dimensional and spanned by the vectors $u^\bullet \in \bbR^5$
\[ u^{\cos}_n = 4\cos(2\pi (n-1) / 5), \quad u^{\sin}_n = \sqrt{8} \sin(2\pi (n-1)/5), \quad 1\le n \le 5. \]
One deduces $a=b=0$ by similar considerations to those in the previous subcase.
\medskip

Having exhausted the possible cases, we conclude that for $v \in \bbR^5$, $\frg_\DSO(v) \geq 2$ implies that $v$ is constant, and thus $\frg_\DSO(5) \leq 1$. Since we have already exhibited a pair of two-parameter families with $\frg_\DSO(v) \geq 1$, we conclude $\frg_\DSO(5)=1$.
\bigskip

Next, we want to show that the presence of (at least one) closed gap implies that $v$ has one of the forms in \ref{item:p5dso+2closed} or \ref{item:p5dso-2closed}. To that end, assume that $v \in \bbR^5$ has a closed gap. Since we work up to translation, we may assume that the closed gap occurs at  $E = 0$. 

\setcounter{case}{0}
\begin{case} \textbf{\boldmath $\Phi_v(0) = +\idty$.}\end{case}
 Applying \eqref{eq:p2Phiv} and \eqref{eq:p3Phiv} with $E =0$, we arrive at
\begin{align*}
0 
& = \Phi_{v_1v_2v_3}(0) - [\Phi_{v_4v5}(0)]^{-1} \\
& = \begin{bmatrix}
        -v_1v_2v_3+ v_1 +v_3  & -v_2 v_3+1 \\
        v_1 v_2 -1 & v_2
    \end{bmatrix}
- \begin{bmatrix}
        -1 & - v_5 \\
        v_4 & v_4v_5 -1
    \end{bmatrix} \\
& = \begin{bmatrix}
-v_1v_2v_3+v_1+v_3 + 1 & -v_2v_3+v_5+1 \\
v_1v_2-v_4-1 & -v_4v_5 + v_2 + 1
\end{bmatrix}.
\end{align*}
Looking at the $(1,1)$ entry and (say) the $(1,2)$ entry, and cyclic permutations of the indices we have 
\begin{align}
\label{eq:p5dso:relation1}
-v_jv_{j+1} v_{j+2} +v_j + v_{j+2} + 1 & = 0 , \quad \forall 1 \le j \le 5 \\ 
\label{eq:p5dso:relation2}
v_jv_{j+1} - v_{j+3}- 1 & =0 \quad \forall 1 \le  j \le 5.
\end{align}
\begin{subcase}
\textbf{\boldmath $v_1v_2 = 1$.}
\end{subcase}

Inserting this assumption into \eqref{eq:p5dso:relation1} with $j=1$ gives 
\begin{equation} 
0 = -v_1 v_2 v_3 + v_1+v_3 + 1, = v_1+1,
\end{equation}
so we see $v_1=-1$, which (by the assumption $v_1v_2=1$) in turn forces $v_2=-1$. Using \eqref{eq:p5dso:relation2} with $j=1$ gives
\[v_4 =v_1v_2-1 = 1-1=0.\]
In a similar way, $v_3 = -v_5-1$. Thus, $v$ has the form $v = (-1,-1-v_5 - 1,0,v_5)$ in this case. Therefore, after a cyclic shift, this has the claimed form with $\lambda = 0$ $\eta = v_5$.

\begin{subcase}
\textbf{\boldmath $v_1v_2 \neq 1$.} 
\end{subcase}
In this case, we can solve \eqref{eq:p5dso:relation1} (with $j=1$) for $v_3$ in terms of $v_1$ and $v_2$ to get
\begin{equation} v_3 = \frac{v_1+1}{v_1v_2-1}. \end{equation}
We get $v_4$ from \eqref{eq:p5dso:relation2} via
\begin{equation} 
v_4= v_1v_2-1.\end{equation}
Finally, we may solve \eqref{eq:p5dso:relation1} with $j=5$ for $v_5$ to get
\begin{equation}
    v_5 = \frac{v_2+1}{v_1 v_2-1}.
\end{equation}

\begin{case} \textbf{\boldmath $\Phi_v(0) =  - \idty$.}\end{case}
This is similar and precisely leads to the other possible form of $v$.

\bigskip

\ref{p5odjm} By reflection symmetry as in Lemma~\ref{lem:odjmReflect}, any nonzero spectral gaps come in matched pairs. Thus, $a \in \bbR_+^5$ has a closed gap at $E \neq 0$ if and only if it also has a closed gap at $-E$. Since $0$ is in the interior of a band, each $a \in \bbR_+^5$ may have $0$, $2$, or $4$ closed gaps. Due to Theorem~\ref{t:isotorusfacts}, if $a$ has $4$ closed gaps, then it is a constant vector, hence reducible. Thus, we automatically have $\frg_\ODJM(5) \leq 2$.

Now, assume $a$ exhibits a closed gap. Scaling $a$ by a positive constant, we may assume that $a$ has closed gaps at $E = \pm 1$. Suppose first that the ``inner'' pair of gaps close. This corresponds to
\begin{equation}
\Psi_a(1) = \idty, \quad \Psi_a(-1)=-\idty.
\end{equation}
As usual, rewrite the first identity as $\Psi_{a_1a_2a_3}( 1) - [\Psi_{a_4a_5}(1)]^{-1}=0$, leading us to 
\begin{align}
\label{eq:p5odjmPsia123-Psia45}
    0 
    = \Psi_{a_3a_2a_1}(1)-[\Psi_{a_4a_5}(1)]^{-1}
     = \begin{bmatrix}
        \frac{1-a_1^2-a_2^2}{a_1a_2a_3}+\frac{a_5}{a_4}
        & \frac{a^2_2-1}{a_1a_2a_3}-\frac{1}{a_4a_5} \\[2mm]
        \frac{a_3}{a_1a_2} - \frac{a_1a_3}{a_2} + \frac{a_5}{a_4}
        & -\frac{a_3}{a_1a_2}+ \frac{a_4^2-1}{a_4a_5}.
    \end{bmatrix}
\end{align}
Using the $(2,1)$ entry gives
\begin{equation} \label{eq:p5odjma5}
    a_5  =  \frac{a_3a_4 }{a_2}(a_1-a_1^{-1}).
\end{equation}
Notice that this forces $a_1 > 1$, since $a_5 > 0$ by definition.
On the other hand, subtracting the $(2,1)$ entry from the $(1,1)$ entry (and multiplying by $a_1a_2a_3$ to clear denominators)
yields
\[1-a_1^2-a_2^2-a_3^2+ a_1^2a_3^2= 0.\]
Solving for $a_2$, one has
\begin{equation}\label{eq:p5odjma2}
    a_2  = \sqrt{(a_1^2 - 1)(a_3^2 - 1)}.
\end{equation}
At this point, we note that $a_2 > 0$, so \eqref{eq:p5odjma2} and $a_1>1$ forces $a_3 > 1$ as well.
Substituting \eqref{eq:p5odjma5} and \eqref{eq:p5odjma2} into the $(1,2)$ entry of \eqref{eq:p5odjmPsia123-Psia45} (and multiplying by $a_1^2 a_2^2 a_4 a_5$ to clean up) gives
\begin{align*}
0
& = (a_2^2-1) \frac{a_1a_2a_4a_5}{a_3} - a_1^2a_2^2\\
& = (a_1^2a_3^2 - a_1^2 - a_3^2)a_4^2(a_1^2-1) - a_1^2 (a_1^2-1)(a_3^2-1) \\
& = (a_1^2-1)\Big( a_1^2a_3^2 a_4^2 - a_1^2 a_4^2 -a_3^2 a_4^2  - a_1^2a_3^2 +a_1^2 \Big).
\end{align*}
As discussed above, $a_1^2-1\neq 0$ and hence the second factor vanishes. Thus,
solving for $a_1$, we have
\begin{align} \label{eq:p5odjma1}
    a_1  = \frac{a_3 a_4}{\sqrt{(a_3^2-1)(a_4^2-1)}}.
\end{align}
Since we deduced $a_3>1$ earlier, \eqref{eq:p5odjma1} forces $a_4>1$ as well.

Now, we substitute \eqref{eq:p5odjma1} back into \eqref{eq:p5odjma2} and simplify to obtain
\begin{align*}
a_2  
& = \sqrt{(a_1^2 - 1)(a_3^2 - 1)} \\
& = \sqrt{\left(\frac{a_3^2a_4^2}{(a_3^2-1)(a_4^2-1)} - 1\right)(a_3^2 - 1)} \\
& = \sqrt{\frac{a_3^2 + a_4^2-1}{a_4^2-1}}.
\end{align*}
Finally, \eqref{eq:p5odjma5} determines $a_5$:
\begin{align*}
a_5  
& =  \frac{a_3a_4 }{a_2}(a_1-a_1^{-1}) \\
& =  a_3a_4 \sqrt{\frac{a_4^2-1}{a_3^2+a_4^2-1}} \left[\frac{a_3 a_4}{\sqrt{(a_3^2-1)(a_4^2-1)}} -\frac{\sqrt{(a_3^2-1)(a_4^2-1)}}{a_3 a_4} \right] \\
& = \frac{1}{\sqrt{a_3^2 + a_4^2-1}} \left[ \frac{a_3^2 + a_4^2 -1}{\sqrt{a_3^2 - 1}}\right] \\
& = \sqrt{ \frac{a_3^2 + a_4^2-1}{a_3^2 -1} }.
\end{align*}
Thus, the desired characterization holds in this situation after a cyclic shift using $\alpha= a_3$ and $\beta = a_4$. 

If $\Psi_a(\pm 1) = \mp \idty$ (which corresponds to the outer pair of gaps closing), then one can verify by completely similar computations that the desired form holds with $\alpha^2+\beta^2<1$. As usual, checking that the gaps do indeed close for $a$ satisfying one of the given conditions follows from direct computations.
\end{proof}

\begin{remark}
In principle, one could try to leverage the explicit representations for period-5 potentials exhibiting at least one closed gap to show that no more than one gap can close for nonconstant potentials. However, we found the algebraic manipulations to be simpler in the argument that we gave.
\end{remark}

\begin{proof}[Proof of Theorem~\ref{t:6}]

\ref{p6dso} Assume $v \in \bbR^6$ is irreducible. To show $\frg_\DSO(v)\leq 2$, it suffices to prove the following statements.
\begin{claim} \label{claim:p6dso1}
If both gaps with $\tr\,\Phi=+2$ close, then $v$ has the form 
\begin{equation} \label{eq:p6:va00a00}
v=(a,0,0,-a,0,0), \quad a \neq 0,
\end{equation} 
up to an additive constant and a cyclic shift.
\end{claim}
\begin{claim} \label{claim:p6dso2}
If $v$ has the form \eqref{eq:p6:va00a00}, then no gaps with $\tr\,\Phi=-2$ may close
\end{claim}
\begin{claim} \label{claim:p6dso3}
    No more than one gap with $\tr\,\Phi=-2$ may close.
\end{claim} 

\begin{claimproof}[Proof of Claim~\ref{claim:p6dso1}]
    Assume $\Phi_v(E_1)= \Phi_v(E_2) = \idty$ for $E_1 \neq E_2$. We will show that this implies either that $v$ is reducible (contrary to our assumption) or that $v$ has the form \eqref{eq:p6:va00a00} up to an additive constant and a cyclic shift.

To that end, we may shift $v$ by a constant and assume that $E_1=-E_2$. From \eqref{eq:p3Phiv}, we note that for $E=E_k$, we have
\begin{align}
\nonumber
0
    & = [\Phi_{v_1v_2v_3}(E)-\Phi_{v_4v_5v_6}(E)^{-1}]_{12} \\
    \label{eq:p6two+2gapscloserel1}
    & = 2-2E^2+(v_2+v_3+v_5+v_6)E-v_2v_3-v_5v_6.
\end{align}
Similar to the argument for $p=5$, apply \eqref{eq:p6two+2gapscloserel1} with $E=E_1,E_2$, subtract the results, and cyclically permute the indices to see that $v$ satisfies
\begin{align}
    \label{eq:p6two+2gapsclosesum1}
    v_1+v_2 + v_4 + v_5 = 0 \\
    \label{eq:p6two+2gapsclosesum2}
    v_2+v_3 + v_5 + v_6 = 0 \\
    \label{eq:p6two+2gapsclosesum3}
    v_1+v_3 + v_4 + v_6 = 0 .
\end{align}
Adding \eqref{eq:p6two+2gapsclosesum1} to \eqref{eq:p6two+2gapsclosesum2} and subtracting \eqref{eq:p6two+2gapsclosesum3} gives $v_5 = -v_2$. In a similar way, we also deduce
\[v_4=-v_1,\quad v_6 = -v_3.\]
Substituting this into \eqref{eq:p6two+2gapscloserel1}, we have $2-2E^2-2v_2v_3=0$, leading to
\begin{equation} \label{eq:p6:E21-v2v3}
E^2 =1-v_2v_3. \end{equation}
Using this, (again taking $E = E_k$) we can substitute and simplify
\begin{align}
    \nonumber
    0
    & = [\Phi_{v_1v_2v_3}(E)-\Phi_{v_4v_5v_6}(E)^{-1}]_{11} \\
    \nonumber
    & =(E-v_1)(E - v_2)(E - v_3) - (E - v_1) - (E-v_3) +(E+v_2) \\
     & = E^3-(v_1+v_2+v_3)E^2+c_1E+v_1+v_2+v_3-v_1v_2v_3, 
\end{align}
where $c_1$ is independent of $E$.
Evaluating this expression for $E=E_1,E_2$, adding the results, dividing by two, and using \eqref{eq:p6:E21-v2v3} gives
\begin{align}
    \nonumber
    0
    & = -(v_1+v_2+v_3)(1-v_2v_3)+v_1+v_2+v_3-v_1v_2v_3 \\
    & = v_2v_3(v_2+v_3).
\end{align}
There are thus three possibilities: $v_2=0$, $v_3=0$, or $v_3=-v_2$.
\begin{case}
    \textbf{\boldmath $v_2=0$.}
    \end{case}

On account of \eqref{eq:p6:E21-v2v3}, this gives $E^2=1$. Taking $E=1$ and using the relations obtained so far gives us
    \[0
    = \Phi_{v_1v_2v_3}(1)-[\Phi_{v_4 v_5 v_6}(1)]^{-1}
    = \begin{bmatrix}
        v_1 v_3 & 0 \\ 0 & -v_1v_3.
    \end{bmatrix}.
    \]
    This forces $v_1=0$ or $v_3=0$, either of which implies that $v$ has the form \eqref{eq:p6:va00a00} up to a cyclic shift.

\begin{case}
    \textbf{\boldmath $v_3=0$.}
    \end{case}
    Similar to the case $v_2=0$, this forces $v$ to be of the form \eqref{eq:p6:va00a00} up to a cyclic shift.
    \begin{case}
    \textbf{\boldmath $v_3= -v_2$.}
    \end{case}

    Adding the $(1,2)$ and $(2,1)$ entries of the matrix $\Phi_{v_1v_2v_3}(E_k) - [\Phi_{v_4v_5v_6}(E_k)]^{-1}$ gives us $v_2^2+v_1v_2=0$. The case $v_2=0$ leads again to the form \eqref{eq:p6:va00a00}, while the case $v_2=-v_1$ forces $v$ to have the form $v=(a,-a,a,-a,a,-a)$, which is reducible.\end{claimproof} 
    
\bigskip
\begin{claimproof}[Proof of Claim~\ref{claim:p6dso2}]
    Consider $v$ of the form \eqref{eq:p6:va00a00} for some $a \neq 0$ and observe that it has exactly two closed gaps. Indeed, one can check directly that
\[ \Phi_v(1) =\Phi_v(-1)= \idty. \]
However, $\Phi_v(E)\neq -\idty$ for all $E$. Indeed, if $\Phi_v(E)= -\idty$, we have
\begin{align*}
    0
    & = [\Phi_{v_1v_2v_3}(E) + [\Phi_{v_4 v_5 v_6}(E)]^{-1}]_{21} \\
    & = 
        -2aE.
\end{align*}
Since $a \neq 0$, this forces $E=0$. But then we compute
\begin{align*}
0
    & = \Phi_{v_1v_2v_3}(0) + [\Phi_{v_4 v_5 v_6}(0)]^{-1} \\
    & = \begin{bmatrix}
        a
        & 0 \\
        0 & -a
    \end{bmatrix} \\
& \neq 0,
    \end{align*}
    a contradiction.\end{claimproof}
\bigskip

\begin{claimproof}[Proof of Claim~\ref{claim:p6dso3}]
     Suppose $\Phi_v(E_1)=\Phi_v(E_2)= -\idty$ for $E_1\neq E_2$. Observe then that
    \begin{align}
        \nonumber
        0
        & = [\Phi_{v_1v_2v_3}(E_k) + [\Phi_{v_4 v_5 v_6}(E_k)]^{-1}]_{21} \\
        \nonumber
        & = (E_k - v_1)(E_k - v_2) - 1 - \big((E_k - v_4)(E_k - v_5) - 1 \big) \\
        \label{eq:p6:trace-2closedv2v3-v5v6}
        & = (v_4+v_5 - v_1 - v_2)E_k + v_1v_2 - v_4v_5.
    \end{align}
    Since this holds for $E_1$ and $E_2$, the coefficient of $E_k$ must vanish. Permuting the variables cyclically gives us
    \begin{align}
    \label{p6:trace-2closedv1+v2-v4-v5}
        v_1+v_2 - v_4 - v_5 & = 0 \\
    \label{p6:trace-2closedv6+v1-v3-v4}
        v_6+v_1 - v_3 - v_4 & = 0.
    \end{align}
Substituting back into \eqref{eq:p6:trace-2closedv2v3-v5v6} gives
\begin{align}
    \nonumber
    0 
    & = v_1v_2 - v_4v_5 \\
    \nonumber
    & = v_1v_2 - v_4(v_1+v_2 - v_4) \\
    & = (v_1-v_4)(v_2-v_4).
\end{align}
This forces $v_1=v_4$ or $v_2=v_4$.
\setcounter{case}{0}
\begin{case}
\textbf{\boldmath $v_1=v_4$.} 
\end{case}

Together with \eqref{p6:trace-2closedv1+v2-v4-v5} and \eqref{p6:trace-2closedv6+v1-v3-v4}, this gives
\begin{align*}
    v_5 & = v_1+v_2-v_4 = v_2 \\
    v_6 & = v_3+v_4-v_1 = v_3,
\end{align*}
contrary to irreducibility of $v$.

\begin{case}
    \textbf{\boldmath $v_2=v_4$.}
\end{case}

Together with \eqref{p6:trace-2closedv1+v2-v4-v5}, we get $v_1=v_5$. Putting this together with $\Phi_v(E_k)=-\idty$, we arrive at
\begin{align*}
    0
    & = [\Phi_{v_1v_2v_3}(E_k) + [\Phi_{v_4v_5v_6}(E_k)]^{-1} ]_{12} \\
    & = \big( -(E_k -v_2)(E_k -v_3)+1 \big) - \big( -(E_k -v_5)(E_k -v_6)+1 \big) \\
    & = -v_2v_3+v_1(v_3+v_2 - v_1) \\
    & = -(v_1-v_2)(v_1-v_3).
\end{align*}
If $v_1=v_2$, then (recalling $v_5=v_1$ and $v_4=v_2$), \eqref{p6:trace-2closedv6+v1-v3-v4} implies $v_3=v_6$ and thus $v$ has the form $aabaab$, which is reducible. Similarly, if $v_1=v_3$, then $v$ has the form $ababab$, which is again reducible.\end{claimproof}
\medskip

Combining Claim~\ref{claim:p6dso1}, \ref{claim:p6dso2}, and \ref{claim:p6dso3}, we obtain $\frg_\DSO(6) \leq 2$. Furthermore, since we have already shown $\frg_\DSO(a,0,0,-a,0,0)=2$ for $a \neq 0$, we deduce $\frg_\DSO(6)=2$.
\bigskip

\ref{p6odjm}.    Note that $\frg_\ODJM(6) \leq 3$. Indeed, we know that $\frg_\ODJM(a) \leq 4$ for any irreducible $a \in \bbR_+^6$. If $\frg_\ODJM(a)=4$, this (by symmetry) means that all non-central gaps close, which forces $a$ to be reducible by Theorem~\ref{t:isotorusfacts}. Thus $\frg_\ODJM(6) \leq 3$ follows.

Let us assume that $a \in \bbR_+^6$ has three closed gaps. As usual, we can scale and assume that the closed gaps occur at $0$ and $\pm 1$. Let us consider the case in which the ``inner'' pair of non-central gaps close, which corresponds to $\Psi(\pm 1) = \idty$. Since the central gap at zero closes, we may deduce 
\begin{equation} \label{eq:p63gapbalance}
    a_1a_3a_5 = a_2a_4a_6.
    \end{equation}
Recalling 
\[\Psi_{a_1a_2a_3}(1) =
\begin{bmatrix}
        \frac{1-a_1^2-a_2^2}{a_1a_2a_3}
        & \frac{a^2_2-1}{a_1a_2a_3} \\
        \frac{a_3}{a_1a_2} - \frac{a_1a_3}{a_2}
        & -\frac{a_3}{a_1a_2}.
    \end{bmatrix}.
\]
Using this and substituting $a_6 = a_1 a_3 a_5 /(a_2a_4)$, we arrive at
\begin{align*}
    0
    & = [\Psi_{a_1a_2a_3}(1) - [\Psi_{a_4a_5a_6}(1)]^{-1}]_{11}
    -  [\Psi_{a_1a_2a_3}(1) -[\Psi_{a_4a_5a_6}(1)]^{-1}]_{21}\\
    & = \frac{1-a_1^2-a_2^2}{a_1a_2a_3}+ \frac{a_6}{a_4a_5}
    -\left( \frac{a_3}{a_1a_2} - \frac{a_1a_3}{a_2} + \left( \frac{a_6}{a_4a_5} - \frac{a_4a_6}{a_5} \right) \right) \\
    & = \frac{1-a_1^2-a_2^2}{a_1a_2a_3}
    - \frac{a_3}{a_1a_2} + \frac{a_1a_3}{a_2} + \frac{a_1a_3}{a_2} \\
    & = \frac{1-a_1^2 - a_2^2 - a_3^2 + 2a_1^2a_3^2}{a_1 a_2 a_3 }.
\end{align*}
Thus,
\[1-a_1^2-a_2^2 = (1-2a_1^2)a_3^2\]
leading to
\begin{equation} \label{eq:p6odjm:innerclosea3} a_3 = \sqrt{\frac{a_1^2 + a_2^2 -1}{2a_1^2 - 1}}.
\end{equation}
Substituting this back into $\Psi_{a_1a_2a_3}(1) - [\Psi_{a_4a_5 a_6}(1)]^{-1}=0$ and considering the $(1,1)$ and $(1,2)$ entries gives us 
\begin{align*}
    a_1^2 + a_4^2 - 2a_1^2a_4^2 &= 0 \\
    a_2^2 + a_5^2 - 2a_2^2a_5^2 &= 0,
\end{align*}
yielding
\begin{align}
\label{eq:p6odjm:innerclosea4}
    a_4 & = \frac{a_1}{\sqrt{2a_1^2-1}} \\
    \label{eq:p6odjm:innerclosea5}
    a_5 & = \frac{a_2}{\sqrt{2a_2^2-1}}.
\end{align}
Substituting \eqref{eq:p6odjm:innerclosea3}, \eqref{eq:p6odjm:innerclosea4}, and \eqref{eq:p6odjm:innerclosea5} into \eqref{eq:p63gapbalance} gives
\begin{align*}
    a_6
    = \frac{a_1a_3a_5}{a_2a_4}
    = \ \sqrt{\frac{a_1^2 + a_2^2 - 1}{2a_2^2 - 1}}.
\end{align*}
Thus, if $a$ has at least three closed gaps and the pair comes from the inner gaps, $a$ has the form
\begin{equation}\label{eq:p6odjm:3closedaform}
    a = \left(\alpha,\beta, \sqrt{\frac{\alpha^2 + \beta^2-1}{2\alpha^2-1}},
\frac{\alpha}{\sqrt{2\alpha^2-1}},
\frac{\beta}{\sqrt{2\beta^2-1}},
\sqrt{\frac{\alpha^2 + \beta^2-1}{2\beta^2-1}}
\right)
\end{equation}
where $\alpha,\beta>1/\sqrt{2}$, up to a constant multiple. The reader can verify by direct computations that any $a \in \bbR_+^6$ of the form \eqref{eq:p6odjm:3closedaform} with $\alpha,\beta>1/\sqrt{2}$ has closed gaps at $E = 0, \pm 1$. For generic choices of $\alpha$ and $\beta$, the resulting $a$ is irreducible, so we see that $\frg_\ODJM(6) \geq 3$, concluding the argument.
\end{proof}

\section{Bounds for Larger Periods} \label{sec:gen}
\begin{proof}[Proof of Theorem~\ref{t:pgeq4}]
\ref{pgeq4dso} One may use the identities
\begin{equation} \label{eq:Mpowers}
    [M(0)]^2 = [M(1)]^3 = -\idty, \quad [M(-1)]^3 = \idty
\end{equation}
to conclude. Indeed, if $p\geq 7$ is odd, then
    \[ v = (1,1,1,\underbrace{0,0,\ldots,0}_{p-3 \text{ copies}}) \]
    produces a closed gap at $E = 0$ on account of \eqref{eq:Mpowers}. Similarly, if $p \geq 7$ is even, then
    $$v = (1,1,1,-1,-1,-1,\underbrace{0,0,\ldots,0}_{p-6 \text{ copies}})$$
    produces a closed gap at zero.
    \bigskip
    
    \ref{pgeq4dsoub} Assume $v \in \bbR^p$ is irreducible. By an inductive calculation, one can check directly that for any $w \in \bbR^m$,
    \begin{equation} \label{eq:Phi21asymptotics}
    [\Phi_w(E)]_{21} = E^{m-1} - \left( \sum_{j=1}^{m-1} w_j \right) E^{m-2} + O(E^{m-3}).
    \end{equation}

        \setcounter{case}{0}
    \begin{case}
        \textbf{\boldmath $p$ is divisible by $4$.}
    \end{case}
Consider the anti-periodic closed gaps of $v$, which are defined by $\Phi_v(E)=-\idty$, which we may rewrite as
\[ \Phi_{v^+}(E) + [\Phi_{v^-}(E)]^{-1} = 0, \]
    where $v^+ = (v_1,\ldots, v_{\frac{p}{2}})$ and $v^- = (v_{\frac{p}{2}+1} , \ldots, v_p)$.
    On account of \eqref{eq:Phi21asymptotics}, we get
    \begin{equation} \label{eq:Phi21clgapasymptotics}
    0 = [\Phi_{v^+}(E) + [\Phi_{v^-}(E)]^{-1}]_{21} = \left( \sum_{j=1}^{\frac{p}{2}-1} (v_{j + \frac{p}{2}} - v_j ) \right) E^{\frac{p}{2}-2} + O(E^{\frac{p}{2}-3})
    \end{equation}
    for any antiperiodic closed gap $E$; as in previous arguments, this holds for all cyclic permutations of $v$. We claim that there are at most $\frac{p}{2}-2$ antiperiodic closed gaps. Indeed, if there are $\frac{p}{2}-1$ or more antiperiodic closed gaps, then we have by \eqref{eq:Phi21clgapasymptotics} and its cyclic permutations:
\begin{equation} \label{eq:vj+p2-vj=0}
    \sum_{j=1}^{\frac{p}{2}-1} (v_{j+\ell + \frac{p}{2}} - v_{j + \ell} ) = 0, \quad \forall 0 \le \ell \le p-1.
\end{equation}
Denote $\gamma_0=1$, $\gamma_1=0$, $\gamma_\ell = (-1)^{\ell+1}$ for $\ell \geq 2$. Multiplying the $\ell$th equation by $\gamma_\ell$ and summing the results from $\ell = 0$ to $\ell = \frac{p}{2}-1$, we get\footnote{Notice that this uses the assumption that $4$ divides $p$.}
\[ v_1 = v_{\frac{p}{2}+1}. \]
Cyclically permuting, one sees $v_j = v_{j+ \frac{p}{2} }$ for any $j$, contrary to irreducibility of $v$.

Thus, there are at most $\frac{p}{2}-2$ antiperiodic closed gaps. Since there are at most $\frac{p}{2}-1$ periodic closed gaps, it follows that the total number of closed gaps is at most $p-3$ for irreducible $v$.
    
    \begin{case}
        \textbf{\boldmath $p$ is odd.} 
    \end{case}
    In this case, put $m = (p+1)/2$, $v^+ = (v_1,\ldots,v_m)$, and $v^- = (v_{m+1},\ldots,v_p)$. Note that $m-1$ is relatively prime to $p$, which can be seen (for instance) by observing
    \[ p - 2(m-1)=1.\]
    Note that \eqref{eq:Phi21asymptotics} gives
    \[ [\Phi_{v^+}(E) - [\Phi_{v^-}(E)]^{-1}]_{21}
    = E^{m-1} + \left(1 - \sum_{j=1}^{m-1} v_j\right)E^{m-2 } + O(E^{m-3}).\]
    We claim that $v$ has at most $m-2$ closed periodic gaps. Indeed, if $v$ has $m-1$ or more closed periodic gaps, the calculation above (and cyclic permutations as usual) implies
    \[ \sum_{j=1}^{m-1} v_j  = \sum_{j=1}^{m-1} v_{j+1},\]
    leading to $v_1 = v_m$. Cyclically permuting again, we get $v_j = v_{j+m-1}$ for any $j$. Since $m-1$ is relatively prime to $p$, $v$ is constant, a contradiction. 

    In a completely similar fashion, one sees that $v$ has no more than $m-2$ closed gaps satisfying $\Phi = -\idty$. Thus, $v$ has at most
    \[ 2(m-2)=p-3 \]
    gaps, as promised.
    \bigskip

    \ref{pgeq4odjm} If $p \geq 7$ is even, choose any irreducible $a \in \bbR_+^p$ such that
    \begin{equation}
        a_1a_3 \cdots a_{p-1} = a_2a_4 \cdots a_p.
    \end{equation}
    By an inductive calculation, one can prove that
    \begin{equation}
        \Psi_a(0) = (-1)^{p/2}\idty,
    \end{equation}
    so $\frg(p) \geq 1$ for every even $p \geq 4$.

    Next, we consider odd $p \geq 4$. Note that
    \begin{equation}
        B(1,1)^3=  B(\tfrac{1}{\sqrt{2}},1)^4 = -\idty.
    \end{equation}
    Thus, if $ p \geq 7$ is congruent to $3$ modulo $4$, then
    \begin{equation}
        a = (1,1,1,\underbrace{\tfrac{1}{\sqrt{2}},\ldots,\tfrac{1}{\sqrt{2}}}_{(p-3) \text{ copies}})
    \end{equation}
    is an irreducible vector enjoying closed gaps at $E = \pm 1$.

Similarly, if $p \geq 7$ is congruent to $1$ modulo $4$, the vector
\begin{equation}
        a = \Big(2,3,2,\sqrt{3/2},\sqrt{3/2},\underbrace{\tfrac{1}{\sqrt{2}},\ldots,\tfrac{1}{\sqrt{2}}}_{(p-5) \text{ copies}} \Big)
    \end{equation}
    exhibits closed gaps at $E = \pm 1$. 
    
\ref{pgeq4odjmub} Assume $p \geq 3$ is given. We note that the number of closed gaps of $a \in \bbR_+^p$  is at most $p-2$ by Theorem~\ref{t:gjac}. If $p$ is odd, then the number of closed gaps of any $a \in \bbR_+^p$ is \emph{even} by Lemma~\ref{lem:odjmReflect}. Since $p-2$ is odd when $p$ is odd, $\frg_\ODJM(p) \leq p-3$ for odd $p\geq 3$.\footnote{Note this gives an abstract proof of part of Theorem~\ref{t:123}.} If $p \geq 3$ is even, the only way for $p-2$ gaps to close would be for all gaps but the central gap to close. However, this would in turn imply that the spectrum of $\boL_a$ has two components, each of equilibrium measure $1/2$, which in turn would imply that $\boL_a$ has period $2$ by Theorem~\ref{t:isotorusfacts}. Since this is a contradiction for $p \geq 4$, we again have $\frg_\ODJM(p) \leq p-3$ in this case.
\end{proof}

\begin{proof}[Proof of Theorem~\ref{t:multiples}]
    Let $p,k \geq 2$ be given. Choose an irreducible $v \in \bbR^p$ maximizing $\frg_\DSO(p)$, that is, such that
    \[\frg:= \frg_\DSO(v) = \frg_\DSO(p), \]
    and denote the energies corresponding to closed gaps by $E_1,E_2,\ldots,E_\frg$. Let $v^\shifted := \cyclic(v)$ denote the cyclic permutation of $v$. As observed above, we have
    \begin{equation} \label{eq:multiplies:basis}
    \Phi_v(E_j) = \Phi_{v^\shifted}(E_j) \in \{\pm \idty\}, \quad \forall \ 1 \le j \le \frg.\end{equation}
Defining $w\in \bbR^{2kp}$ by 
\[w = \underbrace{vv\cdots v}_{k \text{ copies}} \underbrace{v^\shifted v^\shifted \cdots v^\shifted}_{k \text{ copies}},\]
we claim that $\frg_\DSO(w) \geq \frg+(k-1)p$.
Indeed, \eqref{eq:multiplies:basis} already shows
\begin{equation}
    \Phi_w(E_j) = \idty \ \forall \ 1 \le j \le \frg.
\end{equation}
Moreover, for any $E$ such that 
\begin{equation} \label{eq:trace:chebypoint} \tr(\Phi_v(E)) = 2\cos(\pi j /k)
\end{equation}
for an integer $1 \le j \le k-1$, $\Phi_v(E)$ has linearly independent eigenvectors corresponding to eigenvalues $\exp(i \pi j /k)$ and thus $\Phi_v(E)^k= (-1)^j \idty$. By cyclicity of the trace, one gets the same result for $\Phi_{v^\shifted}(E)$, so one arrives at
\[ \Phi(w,E) = \Phi_{v^\shifted}(E)^k\Phi_v(E)^k = \idty \]
for any $E$ satisfying \eqref{eq:trace:chebypoint}. Since $w$ is irreducible by construction, we arrive at $\frg_\DSO(2kp) \geq \frg_\DSO(w) \geq \frg+(k-1)p$

The proof in the case $\bullet = \ODJM$ is identical up to a change in notation.
\end{proof}

\begin{appendix}
    \section{Isospectral Tori}\label{sec:isotori}
    Here we recall a few important notions from the inverse spectral theory of periodic Jacobi matrices. A Jacobi matrix $\boJ$ is said to be \emph{reflectionless} on the set $\Sigma \subseteq \bbR$ if the diagonal elements of its Green function have purely imaginary boundary values Lebesgue almost everywhere on $\Sigma$, that is, for every $n \in \bbZ$,
    \begin{equation}
        \lim_{\varepsilon \downarrow 0} \Re \, \langle \delta_n, (\boJ - E)^{-1}\delta_n \rangle = 0, \ \text{a.e.\ } E \in \Sigma.
    \end{equation}
It is known that every periodic Jacobi matrix, $\boJ$, is reflectionless on its spectrum, $\spectrum(\boJ)$, and that the spectrum is a \emph{finite gap set}, that is, a union of finitely many nondegenerate closed intervals. 

The inverse theory begins with a finite gap set
\begin{equation} \label{eq:fingapdef}
 \Sigma = [\alpha_1,\beta_1] \cup \cdots \cup [\alpha_m,\beta_m]
\end{equation}
where 
\begin{equation} \label{eq:fingapdef2}
    \alpha_1<\beta_1 < \alpha_2 < \cdots < \beta_m
\end{equation} and asks what reflectionless Jacobi matrices may have that set as their spectrum. More precisely, the \emph{isospectral torus} $\isotorus(\Sigma)$ consists of the set of all bounded Jacobi matrices $\boJ$ such that 
\begin{itemize}
    \item $\spectrum(\boJ) = \Sigma$
    \item $J$ is \emph{reflectionless} on $\Sigma$.
\end{itemize}
That this set is non-empty for any finite-gap set $\Sigma$ (let alone a manifold or torus) is non-trivial. We direct the reader to \cite[Chapter~5]{Simon2011Szego} and \cite[Chapters~8-9]{Teschl2000Jacobi} for detailed discussions.

It turns out that there is a remarkable characterization of exactly which finite-gap sets have isospectral tori consisting of periodic Jacobi matrices. To describe this characterization, let $\rho = \rho_\Sigma$ denote the \emph{equilibrium measure} of $\Sigma$, that is, the unique Borel probability measure on $\Sigma$ that minimizes the energy functional
\begin{equation}
    \energy(\mu):= - \iint \log|x-y|\, d\mu(x)\, d\mu(y).
\end{equation}

\begin{theorem} \label{t:isotorusfacts}
    Let $\Sigma$ be a finite-gap set as in \eqref{eq:fingapdef}--\eqref{eq:fingapdef2} with equilibrium measure $\rho = \rho_\Sigma$. If $\rho$ assigns rational weight to every connected component of $\Sigma$, then every element of $\isotorus(\Sigma)$ is periodic. Moreover, if
    \[ \rho([\alpha_j,\beta_j]) = \frac{p_j}{q_j}, \quad j = 1,2,\ldots,m \]
    with $p_j/q_j$ in lowest terms, then every element of $\isotorus(\Sigma)$ is $q$-periodic, where $q = \lcm(q_1,\ldots,q_m)$. In particular, if $\Sigma$ consists of a single interval, then $\isotorus(\Sigma)$ consists of a single $\boJ$, which then has constant diagonal and off-diagonals.
\end{theorem}

If $\boJ$ is periodic of period $q$, then its spectrum is a finite-gap set whose connected components have rational equilibrium measure, and the measure of each connected component is a multiple of $1/q$. we direct the reader to \cite{Simon2011Szego}, particularly Theorem~5.13.8 and Corollary~5.13.9.
    \begin{theorem} \label{t:gjac}
    For any $p \geq 2$, $\frg_\JAC(p)=p-2$.
    \end{theorem}
    \begin{proof}
    Let $p \geq 2$ be given.
    
    Let us first show that $\frg_\JAC(p) \leq p-2$. Indeed, if $\frg(a,v) = p-1$ for some $(a,v) \in \bbR_+^p \times \bbR^p$, then $\spectrum(\boJ_{a,v})$ is an interval, which implies that $a$ and $v$ are constant vectors by the Borg--Hochstadt theorem (also by Theorem~\ref{t:isotorusfacts}), and hence not irreducible.

    To show that $\frg_\JAC(p) \geq p-2$, choose a set $\Sigma \subseteq \bbR$ with two connected components, $I_1$ and $I_2$ such that the equilibrium measure of $I_1$ is $1/p$ (and consequently the equilibrium measure of $I_2$ is necessarily $(p-1)/p$. By Theorem~\ref{t:isotorusfacts}, every element of the isospectral torus of $\Sigma$ is a $p$-periodic Jacobi matrix. Necessarily, then, any element of this isospectral torus yields $p-1-1=p-2$ closed gaps, as desired.
    \end{proof}

    \section{Complex Jacobi Matrices} \label{sec:cx}
    Let us briefly discuss the reason that we restrict the discussion to strictly positive \emph{off-diagonals} in the present work. Indeed, given $a \in \bbC^p$, $v \in \bbR^p$, we may define the associated Jacobi operator by 
\begin{equation}
    [\boJ\psi](n)= A(n-1)^* \psi(n-1)+ V(n)\psi(n) + A(n)\psi(n+1),
\end{equation}
where $*$ denotes the complex conjugate and $A$ and $V$ are the $p$-periodic extentions of $a$ and $v$ as before. 

First, we insist that $a_n \neq 0$ for every $n$, since, if $a_n = 0$ for some $n$, then the ``bands'' of $\spectrum(\boJ_a)$ degenerate to single points, so the relevant questions become trivial or meaningless.

In this case, the transfer matrices take the form
\[
B(t,s) = \frac{1}{s} \begin{bmatrix}
    t & -1 \\ |s|^2 & 0
\end{bmatrix}\]
and solve
\begin{equation}
    \begin{bmatrix}
        u(n+1) \\ A(n)^*u(n)
    \end{bmatrix}
    = B(z - V(n),A(n))\begin{bmatrix}
        u(n) \\ A(n-1)^*u(n-1)
    \end{bmatrix}
\end{equation}
whenever $Ju=zu$.

To see why we also choose to avoid complexifying $a$, let us consider $a \in \bbC^p$ given by 
\begin{equation}
\omega:=\exp(2\pi i /p), \
    a_n = \omega^n, \quad n=1,2,\ldots,p.
    \end{equation}
The reader can check by direct calculations that this $a$ has closed gaps at
\[ E_j = 2\cos(\pi j /p), \quad 1\le j \le p-1, \]
and thus $\frg_\ODJM(a) = p-1$; in particular, $a$ is irreducible in the sense we discussed elsewhere in the paper, but $\boL_a$ has more closed gaps than the upper bound in \eqref{eq:gJAC} allows. The phenomenon responsible for this is the following: $\boL_a$ is unitarily equivalent to the free Laplacian $\boL_\ones$ via the intertwining operator $[\Lambda\psi](n) = \exp(2\pi i n/p)\psi(n)$. Thus, irreducibility of $a \in (\bbC^*)^p$ should be interpreted as irreducibility \emph{modulo holonomy}. Restricting to $a_n \in \bbR_+$ avoids this technicality altogether.
\end{appendix}

\bibliographystyle{abbrvArxiv}

\bibliography{REU2023bib}

\begin{thebibliography}{10}

\bibitem{Avila2009CMP}
A.~Avila.
\newblock On the spectrum and {L}yapunov exponent of limit periodic
  {S}chr\"{o}dinger operators.
\newblock {\em Comm. Math. Phys.}, 288(3):907--918, 2009.

\bibitem{Borg1946Acta}
G.~Borg.
\newblock Eine {U}mkehrung der {S}turm-{L}iouvilleschen {E}igenwertaufgabe.
  {B}estimmung der {D}ifferentialgleichung durch die {E}igenwerte.
\newblock {\em Acta Math.}, 78:1--96, 1946.

\bibitem{Bourgain2005greensfunction}
J.~Bourgain.
\newblock {\em Green's function estimates for lattice {S}chr\"{o}dinger
  operators and applications}, volume 158 of {\em Annals of Mathematics
  Studies}.
\newblock Princeton University Press, Princeton, NJ, 2005.

\bibitem{CarmonaLacroix1990}
R.~Carmona and J.~Lacroix.
\newblock {\em Spectral theory of random {S}chr\"{o}dinger operators}.
\newblock Probability and its Applications. Birkh\"{a}user Boston, Inc.,
  Boston, MA, 1990.

\bibitem{CFKS}
H.~L. Cycon, R.~G. Froese, W.~Kirsch, and B.~Simon.
\newblock {\em Schr\"{o}dinger operators with application to quantum mechanics
  and global geometry}.
\newblock Texts and Monographs in Physics. Springer-Verlag, Berlin, study
  edition, 1987.

\bibitem{DF2022ESO}
D.~Damanik and J.~Fillman.
\newblock {\em One-dimensional ergodic {S}chr\"{o}dinger operators---{I}.
  {G}eneral theory}, volume 221 of {\em Graduate Studies in Mathematics}.
\newblock American Mathematical Society, Providence, RI, 2022.

\bibitem{DamFilWan2023MANA}
D.~Damanik, J.~Fillman, and C.~Wang.
\newblock Thin spectra and singular continuous spectral measures for
  limit-periodic {J}acobi matrices.
\newblock  \href{https://arxiv.org/abs/2110.10113}{{\ttfamily
  arXiv:2110.10113}}.

\bibitem{Hochstadt1975LAA}
H.~Hochstadt.
\newblock On the theory of {H}ill's matrices and related inverse spectral
  problems.
\newblock {\em Linear Algebra Appl.}, 11:41--52, 1975.

\bibitem{Hochstadt1984LAA}
H.~Hochstadt.
\newblock An inverse spectral theorem for a {H}ill's matrix.
\newblock {\em Linear Algebra Appl.}, 57:21--30, 1984.

\bibitem{PasturFigotin1992}
L.~Pastur and A.~Figotin.
\newblock {\em Spectra of random and almost-periodic operators}, volume 297 of
  {\em Grundlehren der mathematischen Wissenschaften [Fundamental Principles of
  Mathematical Sciences]}.
\newblock Springer-Verlag, Berlin, 1992.

\bibitem{Simon2011Szego}
B.~Simon.
\newblock {\em Szeg{\H{o}}'s theorem and its descendants}.
\newblock M. B. Porter Lectures. Princeton University Press, Princeton, NJ,
  2011.
\newblock Spectral theory for $L^2$ perturbations of orthogonal polynomials.

\bibitem{Teschl2000Jacobi}
G.~Teschl.
\newblock {\em Jacobi operators and completely integrable nonlinear lattices},
  volume~72 of {\em Mathematical Surveys and Monographs}.
\newblock American Mathematical Society, Providence, RI, 2000.

\bibitem{Vandenboom2018CMP}
T.~VandenBoom.
\newblock Reflectionless discrete {S}chr\"{o}dinger operators are spectrally
  atypical.
\newblock {\em Comm. Math. Phys.}, 359(2):499--514, 2018.

\end{thebibliography}

\end{document}